\documentclass[10pt,a4paper]{article}
\pdfoutput=1
\usepackage[utf8]{inputenc}
\usepackage[a4paper, total={6in, 9in}]{geometry}
\usepackage[font=small,labelfont=bf]{caption}
\usepackage{amsmath}

\usepackage{listings}
\usepackage{color}
\definecolor{keywordcolor}{rgb}{0.7, 0.1, 0.1}   
\definecolor{tacticcolor}{rgb}{0.0, 0.1, 0.6}    
\definecolor{commentcolor}{rgb}{0.4, 0.4, 0.4}   
\definecolor{symbolcolor}{rgb}{0.0, 0.1, 0.6}    
\definecolor{sortcolor}{rgb}{0.1, 0.5, 0.1}      
\definecolor{attributecolor}{rgb}{0.7, 0.1, 0.1} 

\usepackage{amsfonts}
\usepackage{amssymb}
\usepackage{amsthm}
\usepackage{tikz-cd}
\usepackage{float}
\usepackage{hyperref}
\usepackage{cleveref}
\author{Robin Khanfir\footnote{
  Sorbonne Université, Campus Pierre et Marie Curie, LPSM, Case courrier 158, 4, place Jussieu, 75252 Paris Cedex 05, France. Email: \texttt{robin.khanfir@sorbonne-universite.fr}
} \and Béranger Seguin\footnote{
  Univ. Lille, CNRS, UMR 8524 - Laboratoire Paul Painlevé, F-59000 Lille, France. Email: \texttt{beranger.seguin@ens.psl.eu}}}
\title{Study of a division-like property}

\usepackage[backend=bibtex,style=alphabetic]{biblatex}
\bibliography{fadrings} 

\newcounter{c}[section]

\newtheorem{theorem}[c]{Theorem}
\newtheorem{corollary}[c]{Corollary}
\newtheorem{lemma}[c]{Lemma}
\newtheorem{definition}[c]{Definition}
\newtheorem{notation}[c]{Notation}
\newtheorem{proposition}[c]{Proposition}
\newtheorem{conjecture}{Conjecture}
\newtheorem{remark}[c]{Remark}

\newcommand{\End}{\mathrm{End}}
\newcommand{\Ker}{\mathrm{Ker}}
\newcommand{\Ima}{\mathrm{Im}}
\newcommand{\eqdef}{\overset{\text{def}}{=}}
\newcommand{\N}{\mathbb{N}}
\newcommand{\Z}{\mathbb{Z}}
\newcommand{\Zz}{{\mathbb{Z}_{\geq 0}}}

\newcommand{\andd}{\hspace{1cm}\text{and}\hspace{1cm}}


\begin{document}
\maketitle{}
\vspace{-0.8cm}
\begin{center}
  \sl To our friend Assil Fadle.
\end{center}

\abstract{
  We introduce a weak division-like property for noncommutative rings: a nontrivial ring is fadelian if for all nonzero $a,x$ there exist $b,c$ such that $x=ab+ca$.
  We prove properties of fadelian rings, and construct examples of such rings which are not division rings, as well as non-Noetherian and non-Ore examples.
  \\ \\
  \textbf{Keywords: } Noncommutative rings $\cdot$ Simple algebras $\cdot$ Ore condition $\cdot$ Ore extension\\
  \textbf{Mathematics Subject Classification (MSC 2010): } 16U20 $\cdot$ 13N10 $\cdot$ 12E15 $\cdot$ 16P40
}

\section{Introduction}

This article is a compilation of results obtained between 2017 and 2019 on questions of noncommutative algebra.
At the time, we were a group of students grouped under the informal humoristic name of \textit{Département de Mathématiques Inapplicables}.
Among others, Maxime Ramzi is to thank for his precious help.

All rings considered are nontrivial, unital, and not assumed to be commutative.
Various notions of weak inversibility have already been considered, notably (strong) von Neumann regularity \autocite{neumann} and unit regularity.
We introduce a new class of rings satisfying a weak form of divisibility.
They are the \emph{fadelian} and \emph{weakly fadelian} rings:
\begin{definition}
  \label{defn:fad-rings}
  A ring $R$ is:
  \begin{itemize}
    \item
      \emph{fadelian} if for any $x \in R$ and any nonzero $a \in R$, there exist $b, c \in R$ such that:
        \[ x = ab + ca; \]
    \item
      \emph{weakly fadelian} if for any nonzero $a \in R$, there exist $b, c \in R$ such that:
        \[ 1 = ab + ca. \]
  \end{itemize}
\end{definition}

We have the following implications:
  \[ R \text{ is a division ring}
    \Rightarrow R \text{ is fadelian}
    \Rightarrow R \text{ is weakly fadelian}. \]

A natural question is whether any of these implications are equivalences, and to construct counterexamples when they are not.
\begin{itemize}
  \item
    In \Cref{sn:general}, we prove that weakly fadelian rings are simple (\Cref{prop:wfad-is-simple}) and integral (\Cref{weakly-fad-is-integral}), and that weakly fadelian Ore rings are fadelian (\Cref{wf-ore-is-fad}).
  \item
    In \Cref{sn:diff-alg}, we study differential algebras.
    We give conditions for these to yield fadelian rings (\Cref{necc-cond-Rdelta-fad} and \Cref{crit-Rdelta-fad-comm}).
    These conditions are satisfied for differentially closed fields, so this gives the first example of a fadelian ring which is not a division ring.
    In \Cref{ex-fad-non-noeth} and \Cref{countable-nonnoeth-fad}, we transform this example into a countable non-Noetherian fadelian ring.
  \item
    In \Cref{sn:formal-laurent-series}, we study Laurent series on fadelian rings.
    They are themselves fadelian (\Cref{laurent-series-fad}) and turn previous examples into an example of a non-Ore fadelian ring (\Cref{laurent-nonnoeth-countable-fadring-not-ore}, \Cref{cor:non-ore-fad}).
\end{itemize}

The results are summarized in the following map:
\begin{figure}[H]
  \centering
  \includegraphics[scale=0.7]{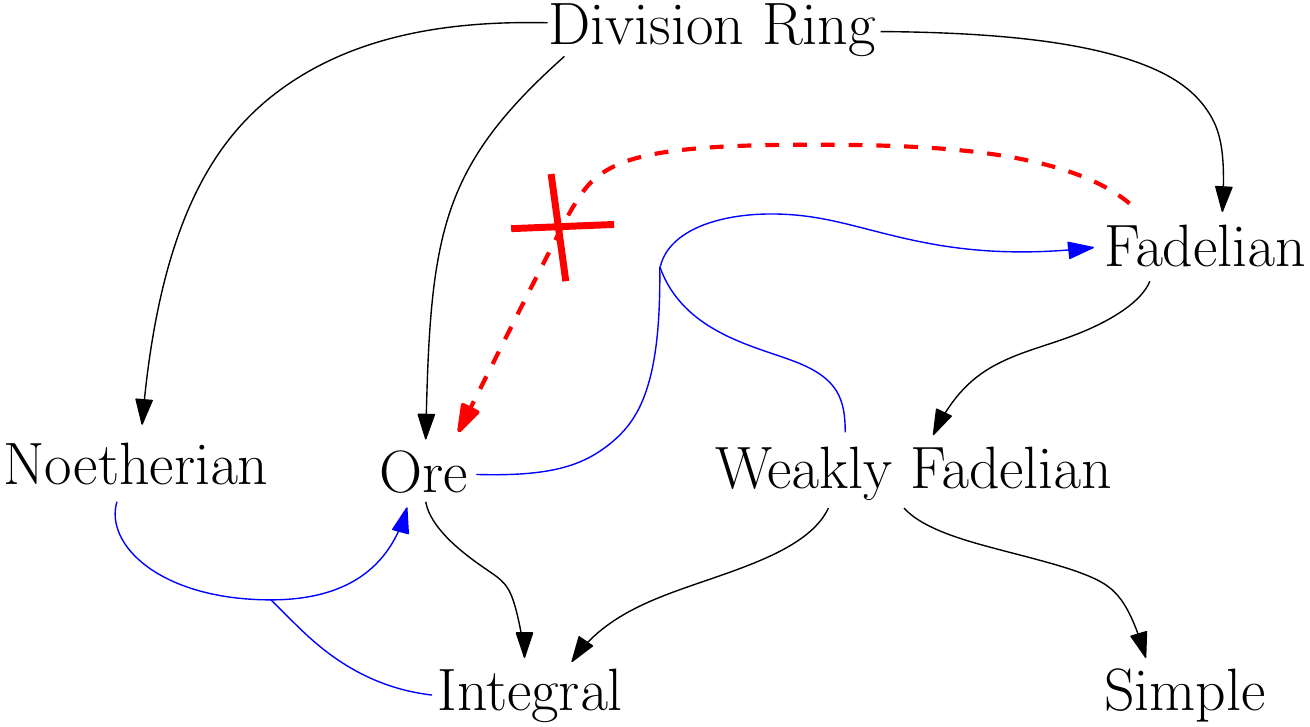}

  \caption{
    Known (non)-implications between the different notions.
    Only a generating set of arrows is drawn.
    The blue arrows are of the form ``$A$ and $B$ imply $C$''.
    The red crossed arrow is known not to be an implication.
  }
\end{figure}

The main open conjecture is the following one:
\begin{conjecture}
  \label{conj:wk-fad-not-fad}
  There is a weakly fadelian ring which is not fadelian.
\end{conjecture}

The authors of this document offer a pizza from \textit{Golosino} to anyone who proves or disproves \Cref{conj:wk-fad-not-fad}.

\section{Properties of weakly fadelian rings}
\label{sn:general}

In this section, $R$ is a weakly fadelian ring.
We prove various properties of $R$.
The proofs of \Cref{weakly-fad-is-integral} and \Cref{wf-ore-is-fad} have been formalized in the Lean proof assistant, cf \Cref{app:formalization}.

A first remark is that if $R$ is commutative, it is a field.
This is generalized by the following proposition:
\begin{proposition}
  \label{prop:wfad-is-simple}
  The ring $R$ is simple.
\end{proposition}

\begin{proof}
  Let $I$ be a nonzero two-sided ideal of $R$.
  Let $a \in I \setminus \{0\}$. 
  Since $R$ is weakly fadelian, there are $b,c \in R$ such that $1 = ab + ca$. We obtain $1 \in I$ and finally $I = R$.
\end{proof}

We now prove the following result:
\begin{theorem}
  \label{weakly-fad-is-integral}
  The ring $R$ is integral.
\end{theorem}

Here, integral means ``containing no zero divisors'', without requiring commutativity.
The proof uses two lemmas:
\begin{lemma}
\label{lemma-first-integral}
  Assume $x,y \in R$ satisfy $xy=yx=0$. Then $x^2=0$ or $y^2=0$.
\end{lemma}

\begin{proof}
  If $x=0$, this is immediate.
  Otherwise, use weak fadelianity to write
    \[ 1 = xb+cx \]
  for some $b,c \in R$.
  Then:
  \begin{align*}
    y^2 & = y \cdot 1 \cdot y \\
    & = y (xb+cx) y \\
    & = yxby + ycxy \\
    & = (yx)by + yc(xy) \\
    & = 0.
  \end{align*}
\end{proof}

\begin{lemma}
  \label{lemma-second-integral}
  Assume $x \in R$ satisfies $x^2=0$.
  Then $x=0$.
\end{lemma}

\begin{proof}
  Assume by contradiction that $x$ is nonzero.
  Write:
  \begin{equation}
    \label{lemma-second-integral-eq-fad}
    1 = xb + cx
  \end{equation} 
  for some $b,c \in R$.
  Notice that:
    \[ cx = cx \cdot 1 = cx(xb+cx) = c(x^2) b + (cx)^2 = (cx)^2. \]
  Similarly, we have $xb = (xb)^2$.
  Since $xb = 1 - cx$, we know that $xb$ and $cx$ commute.
  Hence:
    \[ (xb)(cx) = (cx)(xb) = c (x^2) b = 0. \]
  By \Cref{lemma-first-integral}, we have $(xb)^2=0$ or $(cx)^2 = 0$, and thus $xb=0$ or $cx=0$.
  Assume for example that $xb=0$.
  \Cref{lemma-second-integral-eq-fad} becomes $1 = cx$, so $x$ is invertible, which contradicts $x^2 = 0$.
\end{proof}

We finally prove \Cref{weakly-fad-is-integral}:
\begin{proof}[Proof of \Cref{weakly-fad-is-integral}]
  Let $x,y \in R$ such that $xy = 0$.
  Then $(yx)^2 = y(xy)x = 0$, which implies $yx=0$ by \Cref{lemma-second-integral}.
  By \Cref{lemma-first-integral}, we deduce from $xy=yx=0$ that either $x^2 = 0$ or $y^2 = 0$.
  Applying \Cref{lemma-second-integral} again, we see that either $x$ or $y$ is zero.
\end{proof}

The Ore condition is a well-studied condition, equivalent to the existence of a ring of fractions unique up to isomorphism \autocite{ore}.
It is weaker than Noetherianity (\autocite[Theorem 1]{goldie}, \autocite[Corollary 6.7]{goodearl}).
We recall the definition:
\begin{definition}
  \label{defn:ore}
  The integral ring $R$ is right (resp. left) Ore if any two nonzero right (resp. left) ideals have a nonzero intersection.
\end{definition}

The Ore condition interacts with fadelianity in the following way:
\begin{theorem}
  \label{wf-ore-is-fad}
  If the weakly fadelian ring $R$ is right Ore, it is fadelian.
\end{theorem}

\begin{proof}
  Assume $R$ is right Ore.
  Let $x,a \in R\setminus \{0\}$.
  Since $R$ is right Ore, there exist nonzero elements $b,c \in R$ such that:
    \[ ab = xc. \]

  We have $ca \neq 0$ by \Cref{weakly-fad-is-integral}.
  Since $R$ is weakly fadelian, there exist $k,k' \in R$ such that:
  \[ 1 = cak + k' c a \]

  Finally:
    \[ x = x \cdot 1 = xcak + xk'ca = abak + xk'ca \in aR + Ra. \]

  This proves that $R$ is fadelian.
\end{proof}

In \Cref{cor:non-ore-fad}, we will see that \Cref{wf-ore-is-fad} is not an equivalence.

\section{Fadelianity and differential algebras}
\label{sn:diff-alg}

In this section, we consider differential algebras.
We give conditions (both necessary and sufficient) for \emph{formal differential operator rings} in the sense of \autocite[Chapter 2]{goodearl} to be fadelian.

\subsection{Differential algebras}

\begin{definition}
  Let $k$ be a commutative field and $R$ be a central $k$-algebra.
  A \emph{derivation} of $R$ is a nonzero $k$-linear map $\delta : R \rightarrow R$ such that for all $a,b \in R$:
    \[ \delta(ab) = \delta(a)b + a\delta(b). \]
  We say that $(R, \delta)$ is a differential algebra.
\end{definition}

We fix a differential algebra $(R, \delta)$.
Note that $\delta(1) = \delta(1 \cdot 1) = 2\delta(1)$ and hence $\delta(1)$ is always equal to $0$.
In particular, $\delta$ is non-invertible.

\begin{definition}
  For every $x \in R$, let $\tilde x \in \End(R)$ be the left multiplication endomorphism:
    \[ \tilde x : y \mapsto xy. \]
  We denote by $R[\delta]$ the subalgebra of $\End(R)$ generated by $\delta$ and the endomorphisms $\tilde x$ for $x \in R$.
\end{definition}

An introduction to algebras of formal differential operators can be found in \autocite[Chapter 2]{goodearl}.
These are particular cases of \emph{Ore extensions}.

\begin{remark}
  Since $R[\delta]$ contains the nonzero non-invertible endomorphism $\delta$, it is never a division ring.
\end{remark}

The map $x \mapsto \tilde x$ is an embedding of $R$ into $R[\delta]$.
We see $R$ as a subalgebra of $R[\delta]$, i.e. we identify elements $x \in R$ with their associated left multiplication endomorphism $\tilde x \in R[\delta]$.
To avoid any confusion between $\delta x = \delta \circ {\tilde x}$ and $\delta(x) \in R$, we now solely use the notation $x'$ when evaluating $\delta$ at an element $x\in R$.
By $x^{(n)}$, we mean $\delta^n(x)$.

\begin{remark}
  Let $x \in R$.
  The identity $(xy)' = x'y + xy'$ for $y \in R$ gives the equality $\delta x = x' + x\delta$ of endomorphisms in $R[\delta]$, which can be rewritten using the bracket $[a,b]=ab-ba$:
    \[ [\delta,x] = x'. \]
  In particular, the ring $R[\delta]$ is non-commutative: no $x \in R\setminus\Ker(\delta)$ commutes with $\delta$.
\end{remark}

\subsection{Necessary conditions for $R[\delta]$ to be fadelian}
\label{subsn:necc-cond-diffalg}

\begin{proposition}
  \label{necc-cond-Rdelta-fad}
  Let $(R,\delta)$ be a differential algebra such that $R[\delta]$ is weakly fadelian.
  Then:
  \begin{itemize}
    \item $R$ is a division ring;
    \item $R[\delta]$ is fadelian;
    \item every nonzero element of $R[\delta]$ is surjective as an endomorphism of $R$.
  \end{itemize}
\end{proposition}

This proposition gives necessary conditions on the differential algebra $(R,\delta)$ for $R[\delta]$ to be fadelian.
It also means that this construction cannot give examples of weakly fadelian rings which are not fadelian.

\begin{proof}[Proof of \Cref{necc-cond-Rdelta-fad}]
  Let $a$ be a nonzero element of $R$.
  In $R[\delta]$, write by weak fadelianity:
    \[ 1 = ba\delta + a\delta c \]
  with $b,c \in R[\delta]$.
  Evaluate this equality of endomorphisms at $1 \in R$:
    \[ 1 = b(a \cdot 1') + a \cdot (c(1))' = a \cdot (c(1))'. \]
  The element $(c(1))' \in R$ is an inverse of $a$ in $R$.
  This shows that $R$ is a division ring.

  In particular, $R$ is right Noetherian.
  By \autocite[Theorem 2.6]{goodearl}, $R[\delta]$ is right Noetherian too.
  In particular, $R[\delta]$ is Ore by \autocite[Theorem 1]{goldie}.
  Finally, \Cref{wf-ore-is-fad} implies that $R[\delta]$ is fadelian.

  Now, consider a nonzero element $u \in R[\delta]$.
  Let $a \in A$.
  By fadelianity of $R[\delta]$, write:
    \[ a = d u\delta + u\delta e \]
  with $d,e \in R[\delta]$.
  Evaluate this equality of endomorphisms at $1 \in R$ to obtain:
    \[ a = d(u(1')) + u(e(1)') = u(e(1)') \in \Ima(u). \]
  This proves the surjectivity of $u$.
\end{proof}

\subsection{A criterion for the fadelianity of $R[\delta]$}
\label{subsn:suff-cond-diffalg}

The main theorem of this subsection is the following:
\begin{theorem}
  \label{crit-Rdelta-fad-comm}
  Let $(R,\delta)$ be a differential algebra with $R$ commutative.
  The ring $R[\delta]$ is weakly fadelian if and only if $R$ is a field and every nonzero element of $R[\delta]$ is surjective as an endomorphism of $R$.
\end{theorem}

\Cref{crit-Rdelta-fad-comm} gives the first interesting examples of fadelian rings:
\begin{corollary}
  \label{diff-closed-field-Rdelta-fad}
  If $(k, \delta)$ is a differentially closed field with $\delta \neq 0$, the ring $k[\delta]$ is fadelian and is not a division ring.
\end{corollary}

Differentially closed fields and the existence of differential closures for fields were considered in \autocite{robinson}.
See \autocite[Theorem 6.4.10]{modelth} for a modern approach, and for details about the model theory of differential fields.

The direct implication in \Cref{crit-Rdelta-fad-comm} follows from \Cref{necc-cond-Rdelta-fad}.
In the next paragraphs, \textbf{we assume that $R$ is a differential field and that every nonzero endomorphism in $R[\delta]$ is surjective}, and we introduce tools and results useful for the proof that $R[\delta]$ is fadelian.

\paragraph{Associated polynomial.}

\begin{notation}
  \label{not:pdelta}
  If $P = \sum_{i=0}^d a_i X^i \in R[X]$ is a polynomial, we denote by $P(\delta)$ the element:
  \[ \sum_{i=0}^d a_i \delta^i \in R[\delta]. \]
\end{notation}

\begin{lemma}
  \label{lem:ass-polyn}
  Let $u \in R[\delta]$.
  There is a unique polynomial $P \in R[X]$ such that $u = P(\delta)$.
\end{lemma}

\begin{proof}
\begin{itemize}
  \item
    Using the equality $\delta a = a \delta + a'$ repeatedly, one can make sure all occurrences of $\delta$ are on the right side of products.
    This proves the existence of $P$.

  \item
    Let us prove the uniqueness of $P$.
    By hypothesis, $\delta^n \in R[\delta]\setminus\{0\}$ is surjective.
    We fix elements $x_n \in R$ such that $\delta^n(x_n) = 1$.
    Assume $P(\delta) = Q(\delta)$ for some $P,Q \in R[\delta]$.
    Write:
      \[ P = \sum_{i=0}^d p_i X^i \andd Q = \sum_{i=0}^{d'} q_i X^i. \]
    We show inductively that $p_i = q_i$ for all $i \in \Zz$.
    First, evaluate the equality $P(\delta)=Q(\delta)$ at $1 \in R$ to obtain $p_0 = q_0$.
    Now assume $p_i = q_i$ for all $i<k$.
    We have:
    \begin{align*}
      P(\delta)(x_k) - \sum_{i=0}^{k-1} p_i \delta^i(x_k) & = \sum_{i \geq k} p_i \delta^i(x_k) = p_k \\
      & = Q(\delta)(x_k) - \sum_{i=0}^{k-1} q_i \delta^i(x_k) = \sum_{i \geq k} q_i \delta^i(x_k) = q_k.
    \end{align*}
    So $p_k = q_k$. This concludes the proof.
\end{itemize}
\end{proof}

\paragraph{Degree and Euclidean division.}

If $u \in R[\delta]$, let $P \in R[X]$ be the associated polynomial from \Cref{lem:ass-polyn}.
We define the degree of $u$:
  \[ \theta(u) \eqdef \deg(P). \]

\begin{lemma}
  \label{lem:eucdiv}
  The map $\theta : R[\delta] \rightarrow \Zz \cup \{-\infty\}$ is a Euclidean valuation for which $R[\delta]$ admits left and right Euclidean division.
\end{lemma}

\begin{proof}
  Let $x,y \in R[\delta]$ with $y \neq 0$.
  We prove left Euclidean division by induction on $\theta(x)$.
  If $\theta(x) < \theta(y)$, we simply write $x = 0y + x$.
  Otherwise, write:
    \[ y = \sum_{i=0}^r y_i \delta^i \andd x = \sum_{i=0}^{s} x_i \delta^i \]
  with $r = \theta(y), s = \theta(x)$.
  It follows that $y_r, x_s \neq 0$ and $s \geq r$.
  Define:
    \[ \alpha = x_s y_r^{-1} \delta^{s-r}. \]

  We compute:
  \begin{align*}
    \alpha y & = x_s y_r^{-1} \sum_{i=0}^r \delta^{s-r} y_i \delta^i & \\
    & = x_s y_r^{-1} \sum_{i=0}^r \left ( \sum_{k=0}^{s-r} \binom{s-r}{k} y_i^{(s-r-k)} \delta^{k+i} \right ) & \text{by Leibniz's formula}
  \end{align*}

  In particular, the polynomial $Z$ such that $\alpha y = Z(\delta)$ has leading term:
    \[ x_s y_r^{-1} \binom{s-r}{s-r} y_r^{(s-r-s+r)} \delta^{s-r+r} = x_s y_r^{-1} y_r \delta^s = x_s \delta^s. \]

  Hence, $\theta(x-\alpha y) < \theta(x)$ as the leading terms cancel.
  By induction hypothesis, we may write:
    \[ x - \alpha y = qy+r \]
  with $\theta(r) < \theta(y)$.
  We conclude by writing:
    \[ x = (\alpha + q) y + r. \]

  Right Euclidean division is proved similarly.
\end{proof}

\paragraph{Diagonally dominant systems in differential fields.}

\begin{lemma}
  \label{lemma-domdiag-Rdelta}
  Consider a system of equations of the form:
    \[ \forall j \in \{0, \ldots, r\}, \, \sum_{i=0}^r P_{i,j} (b_i) = x_i \]
  in the indeterminates $b_0, \ldots, b_r$, where $x_i \in R$ and $P_{i,j} \in R[\delta]$.
  Assume moreover:
    \[ \theta(P_{i,i}) > \max_{j \neq i} \theta(P_{i,j}). \]
  Then the system admits a solution $(b_0, \ldots, b_r) \in R^r$.
\end{lemma}

\begin{proof}
  We begin by transforming the system into an equivalent diagonal system with coefficients in $R[\delta]$.
  To do so, we give an algorithm:
  \begin{enumerate}
    \item
      If the system is empty or contains only zero coefficients, stop.
      Otherwise, we can assume $P_{0,0}$ is nonzero by swapping columns and/or rows.
    \item
      For each $i \neq 0$, compute the left Euclidean division using \Cref{lem:eucdiv}:
        \[ P_{i,0} = A_i P_{0,0} + R_{i,0}. \]
      Subtract from the $i$-th row the $0$-th row left-multiplied by $A_i$.
      Do similarly for columns, using right Euclidean divisions and right-multiplication instead.
      We obtain an equivalent system such that $\theta(P_{0,0})$ is stricly larger that the value of $\theta$ at any of the coefficients in the $0$-th row and in the $0$-th column.
    \item
      If there is a row other than the $0$-th whose left coefficient is nonzero, or if there is a column other than the $0$-th whose top coefficient is nonzero, swap the $0$-th row/column with that row/column and go back to step $2$.
      Each time we do so, the value $\theta(P_{0,0})$ decreases, which means the process ends.
      Once this is over, $P_{0,0}$ is the only nonzero coefficient on both the $0$-th row and the $0$-th column.
    \item
      Now, apply steps 1-4 to the subsystem with coefficients $P_{\geq 1, \geq 1}.$
  \end{enumerate}

  After the algorithm has ended, we have an equivalent diagonal system, of the form $\tilde x_i = \tilde P_i(\tilde b_i)$ with $\tilde P_i \in R[\delta]$.
  If this system has a solution, the original system has a solution too.
  Since nonzero endomorphisms in $R[\delta]$ are surjective, we only have to show that all the diagonal coefficients $\tilde P_i$ are nonzero.

  Assume $\forall i, \sum_{j=0}^r P_{i,j} Q_j = 0$ for some nonzero tuple $(Q_j)_{j=0,\ldots,r} \in (R[\delta])^{r+1}$.
  Choose an index $m \in \{0,\ldots,r\}$ such that $\theta(Q_m) \geq \theta(Q_j)$ for all $j$.
  We have:
    \[ \sum_{j \neq m} P_{m,j} Q_j = - P_{m,m} Q_m \]
  and therefore:
  \begin{align*}
    \theta(Q_m) & \leq \max_{j \neq m} (\theta(P_{m,j}) + \theta(Q_j)) - \theta(P_{m,m}) \\
    & < \theta(P_{m,m}) + \theta(Q_m) - \theta(P_{m,m}) = \theta(Q_m).
  \end{align*}

  This is a contradiction.
  Hence the vectors $(P_{i,j})_{i=0,\ldots,r}$ are $r$ independent vectors in $R[\delta]^r$.
  Since the elements $(0,\ldots,0,\tilde P_i,0,\ldots,0) \in R[\delta]^r$ are obtained from them via elementary operations, they are independent too, and thus $\tilde P_i \neq 0$.
\end{proof}

\paragraph{Proof of \Cref{crit-Rdelta-fad-comm}.}

\begin{proof}
  The necessary condition in \Cref{crit-Rdelta-fad-comm} follows from \Cref{necc-cond-Rdelta-fad}.
  Let $(R, \delta)$ be a differential field such that every endomorphism in $R[\delta]$ is surjective.
  Choose a nonzero element $x \in R[\delta]$ and write it as:
    \[ x = \sum_{i=0}^n x_i \delta^i \]
  with $x_n \neq 0$.
  To prove that $R[\delta]$ is weakly fadelian, we want to find $b,c \in R[\delta]$ such that:
    \[ 1 = bx + xc. \]
  i.e. coefficients $b_0, \ldots, b_{n-1}, c_0, \ldots, c_{n-1} \in R$ such that:
    \[ \left ( \sum_{i=0}^{n-1} b_i \delta^i \right ) \left ( \sum_{i=0}^n x_i \delta^i \right ) + \left ( \sum_{i=0}^n x_i \delta^i \right ) \left ( \sum_{i=0}^{n-1} c_i \delta^i \right ) = 1. \]

  Rewrite this as:
  \begin{align*}
    1 & = \sum_{i=0}^{n-1} \sum_{j=0}^n \left ( \sum_{k=0}^i b_i \binom{i}{k} x_j^{(k)} \delta^{i-k+j} \right ) + \left ( \sum_{k=0}^j x_j \binom{j}{k} c_i^{(k)} \delta^{j-k+i} \right ) \\
    & = \sum_{i=0}^{n-1} \sum_{j=0}^n \sum_{k=0}^n \left ( b_i \binom{i}{k} x_j^{(k)} + x_j \binom{j}{k} c_i^{(k)} \right ) \delta^{i+j-k}
  \end{align*}

  By \Cref{lem:eucdiv}, the coefficients in the decomposition of $1$ are unique.
  So, we get a system of equations: for each $d \in \{0, \ldots, 2n-1\}$ we must solve:
  \begin{equation}
    \label{eqn-coeffs-fad-Rdelta}
    \sum_{i=0}^{n-1} \sum_{j=0}^n b_i \binom{i}{i+j-d} x_j^{(i+j-d)} + x_j \binom{j}{i+j-d} c_i^{(i+j-d)} = \delta_{d,0}.
  \end{equation}

  For $d = 2n-1$ we get:
    \[ b_{n-1} x_n + x_n c_{n-1} = 0. \]
  Thus we can express $b_{n-1}$ as a function of $c_{n-1}$.
  Similarly, letting $d=2n-2$ lets one express $b_{n-2}$ as a function of (the derivatives of) $c_{n-2}$ and $c_{n-1}$, and so on until $d=n$.
  We get elements $A_{i,j} \in R[\delta]$ such that:
    \begin{equation}
      \label{eqn:bi-from-cj}
      b_i = \sum_{j=i}^{n-1} A_{i,j}(c_j)
    \end{equation}
  and such that $\theta(A_{i,j}) \leq n-i-1 \leq n-1$.
  Substitute $b_i$ for $\sum_j A_{i,j}(c_j)$ in \Cref{eqn-coeffs-fad-Rdelta} for $0 \leq d \leq n-1$ to obtain a system of $n$ equations of the form:
    \begin{equation}
      \label{eqn:system-fad}
      \sum_{i=0}^{n-1} P_{i,d}(c_i) = \delta_{d,0}.
    \end{equation}

  The expression of $b_i$ in \Cref{eqn:bi-from-cj} involves only derivatives of $c_i,c_{i+1},\ldots,c_{n-1}$ up to the $n-1$-th derivative.
  The only $n$-th derivatives that may appear in $P_{i,d}$ are in the terms:
    \[ x_j \binom{j}{i+j-d} c_i^{(i+j-d)}. \]
  If $d<i$, the binomial coefficient is zero.
  If $d>i$, the derivative is of order $i+j-d < n$.
  If $d=i$, the term obtained for $j=n$ is:
    \[ x_n c_d^{(n)} \]
  which effectively involves an $n$-th derivative with the nonzero coefficient $x_n$.
  This shows:
    \[ \theta(P_{d,d})=n \andd \theta(P_{k,d})\leq n-1 \text{ for } k \neq d. \]
  Hence the system given by \Cref{eqn:system-fad} is diagonally dominant and thus admits a solution by \Cref{lemma-domdiag-Rdelta}.
  This proves that $R[\delta]$ is weakly fadelian.
\end{proof}

\begin{remark}
  We have another sufficient condition that does not require $R$ to be commutative, but requires that all possible nonconstant polynomial differential equations have a solution in $R$, and not only linear ones.
  These are equations that may look something like:
    \[ a (X^{(19)})^2 b X' - c X' d X e X^{(3)} = 0. \]

  We do not include the proof here, since this has not yielded new examples.
\end{remark}

\subsection{A non-Noetherian fadelian ring}
\label{subsn:non-noeth-fad}

\begin{theorem}
  \label{ex-fad-non-noeth}
  There exists a non-Noetherian\footnote{By non-Noetherian, we systematically mean "neither left nor right Noetherian".} fadelian ring.
\end{theorem}

\begin{proof}
  The idea is that fadelianity is a first order property, and thus is preserved by ultrapowers, whereas Noetherianity is not.

  Let $(k,\delta)$ be a differentially closed field as above, with $\delta \neq 0$.
  Then $k[\delta]$ is fadelian by \Cref{crit-Rdelta-fad-comm}.
  Let $\mathcal{U}$ be a non-principal ultrafilter on $\N$.
  The following ring is fadelian by Łoś's theorem (\autocite[Exercise 2.5.18]{modelth}) :
    \[ k[\delta]^{\mathcal{U}} = \left ( \prod_{n\in \N} k[\delta] \right )/\mathcal{U}. \]

  Let $e_n \in k[\delta]^{\mathcal{U}}$ be the coset of:
    \[ \hat e_n = (\underbrace{1_{k[\delta]}, 1_{k[\delta]}, \ldots, 1_{k[\delta]}}_{n}, \delta, \delta^2, \delta^3, \ldots) \in k[\delta]^\N. \]
  Let $I_n$ be the left ideal of $k[\delta]^{\mathcal{U}}$ generated by $e_n$.
  We have:
    \[ \hat e_{n-1} = (\underbrace{1_{k[\delta]}, 1_{k[\delta]}, \ldots, 1_{k[\delta]}}_{n-1}, \delta, \delta, \delta, \ldots) \hat e_n. \]

  So $e_{n+1} \in I_n$, and the sequence of left ideals $(I_n)$ is nondecreasing.
  We prove that it is strictly increasing by contradiction.
  Assume $e_n \in I_{n-1}$ for some $n\geq 0$.
  Then there exists $a \in k[\delta]^\N$ such that:
    \[ \hat e_n \sim a \hat e_{n-1} = \left ( a_1, \ldots, a_{n-1}, a_n \delta, a_{n+1} \delta^2, \ldots \right ). \]

  Consider the map $\psi : k[\delta] \rightarrow \Zz \cup \{+\infty\}$ defined in the following way: $\psi(0)=+\infty$, and otherwise write $x = P(\delta)$ as in \Cref{not:pdelta} and let $\psi(x) \eqdef \min \{ i \mid x_i \neq 0 \}$.
  The uniqueness part of \Cref{lem:ass-polyn} ensures that this is well-defined.
  For an element $\hat x \in k[\delta]^{\N}$, denote by $\psi(\hat x)$ the element of $(\Zz\cup\{+\infty\})^{\N}$ obtained by evaluating $\psi$ coordinatewise.
  Using $\geq$ to notate coordinatewise inequality, we have:
  \begin{align*}
    \psi(a \hat e_{n-1}) & = (\psi(a_0), \ldots, \psi(a_{n-1}), \psi(a_n) + 1, \psi(a_{n+1}) + 2, \ldots) \\
    & \geq (\underbrace{0,\ldots, 0}_{n-1}, 1, 2, 3, 4, \ldots).
  \end{align*}
  On the other hand:
    \[ \psi(\hat e_n) = (\underbrace{0,\ldots,0}_{n-1}, 0,1,2,3, \ldots). \]
  Therefore $\hat e_n$ and $a \hat e_{n-1}$ have finitely many common coefficients, which contradicts $\hat e_n \sim a \hat e_{n-1}$.
  So $(I_n)$ is a strictly increasing sequence of left ideals.
  This contradicts left Noetherianity.
  We prove similarly that $k[\delta]^{\mathcal{U}}$ is not right Noetherian.
\end{proof}

\subsection{A countable non-Noetherian fadelian ring}
\label{subsn:countable-non-noeth-fad}

\begin{lemma}
  \label{lemma-subring-fad-countable}
  Let $R$ be a fadelian ring and $S$ be a subset of $R$.
  Let $\kappa$ be the cardinal $\max(\aleph_0, |S|)$.
  There is a fadelian subring of $R$ of cardinality $\leq \kappa$ which contains $S$.
\end{lemma}

\begin{proof}
  We construct a sequence $R_i$ of subrings of $R$ of cardinality $\leq \kappa$ in the following way:
  \begin{itemize}
    \item
      $R_0$ is the subring of $R$ generated by $S$;
    \item
      Assume we have constructed $R_n$.
      For every couple $x,a \in R_n$ such that $a \neq 0$, choose elements $b_n(x,a)$ and $c_n(x,a)$ in $R$ such that:
        \[ x = ab_n(x,a) + c_n(x,a) a. \]

      We may do so since $R$ is fadelian.
      Let $R_{n+1}$ be the subring of $R$ generated by $R_n$ and the elements $b_n(x,a), c_n(x,a)$ for all pairs $x,a \in R_n$ with $a \neq 0$.
  \end{itemize}

  Finally, define:
    \[ R_{\infty} = \bigcup_{n \geq 0} R_n. \]

  Since $R_{\infty}$ is the increasing union of a countable family of rings of cardinality $\leq \kappa$ containing $S$, it is itself a ring of cardinality $\leq \kappa$ containing $S$.
  To prove that $R_{\infty}$ is fadelian, consider elements $x,a \in R_{\infty}$ with $a \neq 0$.
  There exists $n \in \N$ such that both $x$ and $a$ are in $R_n$.
  In $R_{n+1}$ and therefore in $R_{\infty}$, we have $x = ab_n(x,a) + c_n(x,a)a$.
  This proves that $R_{\infty}$ is fadelian.
\end{proof}

\begin{theorem}
  \label{countable-nonnoeth-fad}
  There exists a countable non-Noetherian fadelian ring.
\end{theorem}

\begin{proof}
  Start with the non-Noetherian fadelian ring $R = k[\delta]^{\mathcal{U}}$ obtained in the proof of \Cref{ex-fad-non-noeth}.
  Let $S$ be the countable subset $\{ u, e_0, e_1, e_2, \ldots \}$ of $R$, where $u$ is the coset of $(\delta, \delta, \delta, \ldots) \in k[\delta]^{\N}$ and $e_n$ is the coset of:
    \[ \hat e_n = (\underbrace{1_{k[\delta]}, \ldots, 1_{k[\delta]}}_{n}, \delta, \delta^2, \delta^3, \ldots) \in k[\delta]^{\N}. \]
  By \Cref{lemma-subring-fad-countable}, there is a countable fadelian subring $R_{\infty}$ of $R$ containing $S$.
  Using elements of $S$, we replicate the proof of \Cref{ex-fad-non-noeth} in $R_{\infty}$: the sequence $(R_{\infty} e_n)_{n \geq 0}$ of left ideals of $R_{\infty}$ is strictly increasing, and similarly for the right ideals $(e_n R_{\infty})_{n \geq 0}$.
  We conclude that $R_{\infty}$ is a countable non-Noetherian fadelian ring.
\end{proof}

\section{Formal Laurent series on fadelian rings}
\label{sn:formal-laurent-series}

In this section, we study formal series over a domain $R$.
We define them in the following way:
\begin{definition}
  The ring $R[[X]]$ is the ring of formal series with coefficients in $R$ where the indeterminate $X$ commutes with elements of $R$, i.e. multiplication is given by:
    \[ \left ( \sum_{n \geq 0} a_n X^n \right ) \left ( \sum_{n \geq 0} b_n X^n \right ) = \sum_{n\geq 0} \left ( \sum_{i=0}^n a_i b_{n-i} \right ) X^n. \]
\end{definition}

If $P = \sum_{n\geq 0} a_nX^n$ is an element of $R[[X]]$, we denote by $P(0)$ the element $a_0 \in R$.
We also define Laurent series over $R$:
\begin{definition}
  The ring $R((X))$ consists of elements which are either $0$ or of the form $X^j P$, where $j \in \Z$, $P \in R[[X]]$ and $P(0) \neq 0$, equipped with the product $(X^j P)(X^k Q) = X^{j+k} (PQ)$.
\end{definition}

This construction is related to fadelianity because it preserves it:
\begin{theorem}
  \label{laurent-series-fad}
  The domain $R$ is fadelian if and only if $R((X))$ is fadelian.
\end{theorem}

\begin{proof}
~
\begin{itemize}
  \item
    First assume that $R((X))$ is fadelian.
    Let $x, a \in R\setminus\{0\}$.
    There are $X^j P, X^k Q \in R((X))$, with $P(0), Q(0) \neq 0$, such that:
      \[ x = X^j P a + a X^k Q. \]
    By multiplying by some $X^r$, one may assume that $r,j,k$ are three nonnegative integers, one of them zero, such that:
      \[ x X^r = X^j P a + a X^k Q . \]

    If $r=0$, then $x = X^j P a + a X^k Q$ in $R[[X]]$ and by evaluating at $0$, we get $x \in Ra+aR$.

    Otherwise, we have $r \geq 1$ and either $j$ or $k$ is zero.
    We assume for example that $j=0$.
    We have:
      \[ x X^r = P a + a X^k Q. \]

    Since $r \geq 1$, we know that $(x X^r)(0)=0$.
    Moreover $(Pa)(0) = P(0)a$ is nonzero because $R$ is integral.
    Hence $(a X^k Q)(0)$ is also nonzero.
    This means that necessarily $k=0$.
    We have the equality:
      \[ x X^r = Pa + aQ. \]
    By evaluating this equality at $0$, we get: $0 = P(0)a + aQ(0)$.
    Hence:
      \[ x X^r = (P-P(0))a + a(Q-Q(0)). \]
    Both $P-P(0)$ and $Q-Q(0)$ cancel at $0$. We can factor $X$ from the equality:
      \[ x X^{r-1} = P_1 a + a Q_1. \]
    Iterate the process to reach:
      \[ x = P_r a + a Q_r. \]
    Finally, evaluate at zero to obtain the desired equality in $R$:
      \[ x = P_r(0) a + a Q_r(0). \]
    This proves that $R$ is fadelian.

  \item
    Now assume that $R$ is fadelian.
    Consider two nonzero elements of $R((X))$ written as $X^j P, X^k Q$ with $j,k \in \Z$, $P, Q \in R[[X]]$ and $Q(0) \neq 0$.

    Write $P = \sum_{n \geq 0} p_n X^n$ and $Q = \sum_{n \geq 0} q_n X^n$.
    Then $q_0$ is a nonzero element of $R$.

    We are searching for sequences $(b_0, b_1, b_2, \ldots)$ and $(c_0, c_1, c_2, \ldots)$ of elements of $R$ such that:
    \begin{equation}
      \label{eqn-fad-laurent}
      \sum_{n \geq 0} p_n X^n
      = \left ( \sum_{n \geq 0} q_n X^n \right ) \left ( \sum_{n \geq 0} b_n X^n \right ) + \left ( \sum_{n \geq 0} c_n X^n \right ) \left ( \sum_{n \geq 0} q_n X^n \right ).
    \end{equation}
    If we find such sequences, then we have the equality in $R((X))$:
      \[ X^j P = (X^k Q) \left ( \sum_{n \geq 0} b_n X^{n+j-k} \right ) + \left ( \sum_{n \geq 0} c_n X^{n+j-k} \right ) (X^k Q) \]
    which shows that $R((X))$ is fadelian.
    
    We prove the existence of $(b_n)$ and $(c_n)$ by induction:
    \begin{itemize}
      \item
        Looking at the constant coefficient in \Cref{eqn-fad-laurent}, we get the equality:
          \[ p_0 = q_0 b_0 + c_0 q_0. \]
        We can fix $b_0, c_0 \in R$ satisfying this equality, because $R$ is fadelian and $q_0 \neq 0$.

      \item
        Assume we have defined $b_0,\ldots,b_{n-1},c_0,\ldots,c_{n-1}$ such that the coefficients in front of $X^i$ are equal in both sides of \Cref{eqn-fad-laurent}, for $i = 0, \ldots, n-1$.
        Now consider the coefficient in front of $X^n$.
        We are trying to solve the equation:
          \[ p_n = \sum_{i=0}^n q_i b_{n-i} + c_{n-i} q_i \]
        which can be rewritten as:
          \[ p_n - \sum_{i=1}^n (q_i b_{n-i} + c_{n-i} q_i) = q_0 b_n + c_n q_0. \]
        We can fix $b_n, c_n \in R$ satisfying this equality, because $R$ is fadelian and $q_0 \neq 0$.
    \end{itemize}
\end{itemize}
\end{proof}

To construct a non-Ore fadelian ring (\Cref{cor:non-ore-fad}), the main ingredient is the following theorem:
\begin{theorem}
  \label{laurent-nonnoeth-countable-fadring-not-ore}
  Assume $R$ is a countable fadelian ring which is not right (resp. left) Noetherian.
  Then $R((X))$ is a fadelian ring which is not right (resp. left) Ore.
\end{theorem}

A more general version holds (if $R$ is a non-Noetherian countable simple domain, then $R((X))$ is not Ore) but we prove the weaker version.
\Cref{laurent-nonnoeth-countable-fadring-not-ore} and \Cref{countable-nonnoeth-fad} directly imply:
\begin{corollary}
  \label{cor:non-ore-fad}
  There exists a fadelian ring which is neither right nor left Ore.
\end{corollary}

\begin{proof}[Proof of \Cref{laurent-nonnoeth-countable-fadring-not-ore}]
  We focus on right ideals, as the other case is dual.
  The ring $R((X))$ is fadelian by \Cref{laurent-series-fad}.
  Since $R$ is countable and not right Noetherian, there are:
  \begin{itemize}
    \item
      a bijective enumeration $(a_0, a_1, a_2, \ldots)$ of all elements of $R$, with $a_0 = 0$;
    \item
      a strictly increasing sequence of right ideals of $R$:
        \[ 0 = \tilde I_0 \subsetneq \tilde I_1 \subsetneq \tilde I_2 \subsetneq \ldots. \]
  \end{itemize}

  Choose for every $n \geq 0$ an element $b_n \in \tilde I_{n+1} \setminus \tilde I_n$ and define the right ideal:
    \[ I_n \eqdef b_0 R + \ldots + b_n R (\subseteq \tilde I_{n+1}). \]

  Let $n\geq 1$.
  We shall prove that there is a $c_n \in R$ such that $b_n c_n a_n \not\in I_{n-1}$.
  Assume by contradiction that there is no such $c_n$.
  Then $b_n R a_n \subseteq I_{n-1}$.
  Write:
    \[ 1 = \gamma_n a_n + a_n \delta_n. \]
  for some $\gamma_n, \delta_n \in R$.
  Then:
    \[ b_n = b_n \cdot 1 = b_n ( \gamma_n a_n + a_n \delta_n ) = \underbrace{b_n \gamma_n a_n}_{\in b_n R a_n} + ( \underbrace{b_n \cdot 1 \cdot a_n}_{\in b_n R a_n} ) \delta_n \in I_{n-1} \subseteq \tilde I_n. \]


  This contradicts the choice of $b_n$ as an element of $\tilde I_{n+1} \setminus \tilde I_n$.
  So we may choose a sequence $c_1, c_2, \ldots$ of elements of $R$ such that $b_n c_n a_n \not\in I_{n-1}$.
  We also define $c_0 = 1$.

  Define the set $J_n \eqdef \left \{ x \in R \,\middle\vert \, b_n c_n x \in I_{n-1} \right \}$.
  For $n \geq 1$, the set $J_n$ is a right ideal of $R$ which does not contain $a_n$.
  Since the elements $(a_n)_{n \geq 1}$ form an exhaustive enumeration of $R \setminus \{0\}$, we have:
    \[ \bigcap_{n \geq 1} J_n = 0. \]

  Now consider the two following nonzero elements of $R((X))$:
    \[ A = \sum_{n \geq 0} b_n c_n X^n \andd B = A + b_0. \]

  To prove that $R((X))$ is not right Ore, it suffices to prove that $AT=BS$ implies $T=S=0$.
  Assume by contradiction that $AT=BS$, where $T$ and $S$ are nonzero elements which we write as:
    \[ T = \sum k_n X^n \andd S= \sum l_n X^n. \]

  By multiplying by some power of $X$, we may assume that $k_n = l_n = 0$ for negative $n$ and that either $k_0 \neq 0$ or $l_0 \neq 0$.
  Look at the constant coefficient in the equality $AT = BS$ to obtain:
    \[ b_0 k_0 = 2 b_0 l_0. \]

  Since $R$ is integral and $b_0 \neq 0$, we have $k_0 = 2 l_0$.
  Now look at the coefficient in front of $X^i$ in the equality $AT = BS$:
    \[ b_i c_i k_0 + \sum_{j=0}^{i-1} b_j c_j k_{i-j} = b_i c_i l_0 + b_0 l_i + \sum_{j=0}^{i-1} b_j c_j l_{i-j}. \]

  Substitute $k_0$ by $2l_0$ in this equality and isolate the term $b_i c_i l_0$ to obtain:
    \[ b_i c_i l_0 = \sum_{j=0}^{i-1} b_j c_j (l_{i-j} - k_{i-j}) + b_0 l_i. \]

  Hence, $b_i c_i l_0$ belongs to the ideal $I_{i-1}$.
  This means that $l_0$ belongs to the ideal $J_i$ for all $i \geq 1$.
  As we have shown, $\bigcap_{i \geq 1} J_i = 0$.
  This implies $l_0=k_0=0$, which is a contradiction.
\end{proof}

\begin{remark}
  In this article, we have used the axiom of choice crucially multiple times.
  We have used it to obtain a differentially closed field to apply \Cref{diff-closed-field-Rdelta-fad} to, and to obtain a non-principal ultrafilter in the proof of \Cref{ex-fad-non-noeth}.

  Nevertheless, Maxime Ramzi has observed that \Cref{cor:non-ore-fad} holds in ZF by the following argument: if $V$ is a model of ZF and $L$ is its constructible universe, then $L$ is a model of ZFC and thus it proves that the first-order theory of non-Ore fadelian rings has a model.
  This model is also a model in $V$ (we omit the verifications).
  We conclude by completeness that a choice-free proof of \Cref{cor:non-ore-fad} exists.

  Funnily, this argument does not work for \Cref{ex-fad-non-noeth}, because the theory of non-Noetherian rings is not first-order.
  But since \Cref{cor:non-ore-fad} implies \Cref{ex-fad-non-noeth}, it is still true that \Cref{ex-fad-non-noeth} holds in ZF.
  
  This also means that the answer to \Cref{conj:wk-fad-not-fad} can not require the axiom of choice.
\end{remark}

\printbibliography

@article{ore,
  title={Linear Equations in Non-Commutative Fields},
  author={Oystein Ore},
  journal={Annals of Mathematics},
  year={1931},
  volume={32},
  pages={463}
}

@article{robinson,
  title={On the concept of a differentially closed field},
  author={Abraham Robinson},
  journal={Bull. Res. Council Israel},
  year={1959},
  pages={113-128}
}

@article{goldie,
    author = {Alfred W. Goldie},
    title = "{The Structure of Prime Rings Under Ascending Chain Conditions}",
    journal = {Proceedings of the London Mathematical Society},
    volume = {s3-8},
    number = {4},
    pages = {589-608},
    year = {1958},
    issn = {0024-6115},
    doi = {10.1112/plms/s3-8.4.589}
}

@book{modelth,
author = {David Marker},
year = {2002},
title = {Model Theory: An Introduction},
publisher = {Springer New York, NY},
isbn = {978-0-387-98760-6},
doi = {10.1007/b98860}
}

@book{goodearl,
shorthand = {GW04},
place={Cambridge},
series={London Mathematical Society Student Texts}, title={An Introduction to Noncommutative Noetherian Rings}, DOI={10.1017/CBO9780511841699}, publisher={Cambridge University Press}, author={\relax{Kenneth R.} Goodearl and Robert B. Warfield, Jr.},
year={2004},
collection={London Mathematical Society Student Texts}}

@article{neumann,
 ISSN = {00278424},
 URL = {http://www.jstor.org/stable/86608},
 author = {John Von Neumann},
 journal = {Proceedings of the National Academy of Sciences of the United States of America},
 number = {12},
 pages = {707--713},
 publisher = {National Academy of Sciences},
 title = {On Regular Rings},
 volume = {22},
 year = {1936}
}

\newpage
\appendix
\section{Formalization of results from \Cref{sn:general} in Lean 3.48}
\label{app:formalization}

\begin{lstlisting}[escapeinside={(@}{@)}]
import algebra.ring
import tactic.nth_rewrite
import tactic.noncomm_ring

-- Thanks to the Lean Zulip for their help, especially to the following people: Riccardo Brasca, Eric Wieser, Ruben Van de Velde, Patrick Massot

-- Fadelian, weakly fadelian rings (@[\Cref{defn:fad-rings}]@)
class fadelian (R : Type*) [ring R] : Prop :=
  (prop : ∀(x:R), ∀(a:R), (a ≠ 0) → (∃(b:R), ∃(c:R), x=a*b+c*a))
class weak_fadelian (R : Type*) [ring R] : Prop :=
  (prop : ∀(a:R), (a ≠ 0) → (∃(b:R), ∃(c:R), 1=a*b+c*a))

-- Fadelian rings are weakly fadelian
instance fadelian.to_weak_fadelian {R : Type*} [ring R] [fadelian R] :
  weak_fadelian R := begin
    apply weak_fadelian.mk,
    exact fadelian.prop 1,
  end

-- Integral rings
class integral (R : Type*) [ring R] : Prop :=
  (prop : ∀(x:R), ∀(y:R), (x*y=0) → (x=0) ∨ (y=0))

-- Left Ore rings (@[\Cref{defn:ore}]@)
class left_ore (R : Type*) [ring R] : Prop :=
  (prop : ∀(x:R), ∀(y:R), (x≠0) → (y≠0) → ∃(a:R),∃(b:R),(a≠0) ∧ (b≠0) ∧ (a*x=b*y))

-- In a weakly fadelian ring, (@$xy=yx=0 \Rightarrow x^2=0 \text{ or } y^2=0$@) (@[\Cref{lemma-first-integral}]@)
lemma lem_integral_1 {R :Type*} [ring R] [weak_fadelian R]
  (x:R) (y:R) (xy_zero : x*y=0) (yx_zero : y*x=0) :
  (x*x=0) ∨ (y*y=0) :=
  begin
    cases (em (x=0)) with x_zero x_nonzero,
    
    have xx_zero : x*x=0 := by rw [x_zero, mul_zero],
    left, exact xx_zero,

    obtain ⟨b,c,d⟩ := weak_fadelian.prop x x_nonzero,
    have yy_zero : (y*y = 0) := begin
      nth_rewrite 0 ← mul_one y,
      rw [d, mul_add, add_mul],
      assoc_rewrite yx_zero,
      assoc_rewrite xy_zero,
      noncomm_ring,
    end,
    right, exact yy_zero,
end

-- In a weakly fadelian ring, (@$x^2=0 \Rightarrow x=0$@) (@[\Cref{lemma-second-integral}]@)
lemma lem_integral_2 {R :Type*} [ring R] [weak_fadelian R]
  (x : R) (xx_zero : x*x=0) :
  x=0 :=
  begin
    classical,
    by_contradiction x_nonzero,
    obtain ⟨b, c, d⟩ := weak_fadelian.prop x x_nonzero,

    have cx_eq_cxcx : c*x=c*(x*c)*x := begin
      rw [← mul_one (c*x), d, mul_add],
      assoc_rewrite xx_zero,
      noncomm_ring,
    end,
    have xb_eq_xbxb : x*b=x*(b*x)*b := begin
      rw [← one_mul (x*b), d, add_mul],
      assoc_rewrite xx_zero,
      noncomm_ring,
    end,

    have cxxb_zero : (c*x)*(x*b) = 0 :=
      begin assoc_rewrite xx_zero, noncomm_ring end,

    have xb_from_cx : x*b = 1 - c * x
      := eq_sub_of_add_eq (eq.symm d),
    have xb_cx_commute : (x*b)*(c*x) = (c*x)*(x*b)
      := by rw [xb_from_cx, sub_mul, one_mul, mul_sub, mul_one],
    have xbcx_zero : (x*b)*(c*x) = 0
      := (eq.congr rfl cxxb_zero).mp xb_cx_commute,

    have xbxb_or_cxcx_zero : ((x*b)*(x*b)=0) ∨ ((c*x)*(c*x))=0
      := lem_integral_1 (x*b) (c*x) xbcx_zero cxxb_zero,

    have one_eq_zero : ((0:R)=(1:R)) := begin
      cases xbxb_or_cxcx_zero with xbxb_zero cxcx_zero,

      assoc_rewrite (eq.symm xb_eq_xbxb) at xbxb_zero,
      rw [xbxb_zero, zero_add] at d,
      have ccxx_one : (c*c)*(x*x) = 1
        := by rw [← mul_assoc, mul_assoc c c x, ← d, mul_one, ← d],
      rw [xx_zero, mul_zero] at ccxx_one,
      apply ccxx_one,

      assoc_rewrite (eq.symm cx_eq_cxcx) at cxcx_zero,
      rw [cxcx_zero, add_zero] at d,
      have xxbb_one : (x*x)*(b*b) = 1
        := by rw [← mul_assoc, mul_assoc x x b, ← d, mul_one, ← d],
      rw [xx_zero, zero_mul] at xxbb_one,
      apply xxbb_one,
    end,

    have x_one : (x=x*1) := by rw mul_one,
    rw [← one_eq_zero, mul_zero] at x_one,
    exact x_nonzero x_one,
  end

-- Weakly fadelian rings are integral (@[\Cref{weakly-fad-is-integral}]@)
instance weak_fadelian.to_integral {R :Type*} [ring R] [weak_fadelian R] :
  integral R :=
begin
  apply integral.mk, intro x, intro y, intro xy_zero,
  have yx_zero : y*x=0 := begin
    apply lem_integral_2 (y*x),
    assoc_rewrite xy_zero,
    noncomm_ring,
  end,
  cases (lem_integral_1 x y xy_zero yx_zero)
    with xx_zero yy_zero,
  left, exact lem_integral_2 x xx_zero,
  right, exact lem_integral_2 y yy_zero,
end

-- Weakly fadelian left Ore rings are fadelian (@[\Cref{wf-ore-is-fad}]@)
theorem left_ore_weak_fadelian_is_fadelian {R :Type*} [ring R] [left_ore R] [weak_fadelian R] :
  fadelian R :=
begin
  have H : integral R := by apply weak_fadelian.to_integral,

  apply fadelian.mk, intro x, intro a, intro a_nonzero,
  cases (em (x=0)) with x_zero x_nonzero,

  existsi (0:R), existsi (0:R), rw x_zero, noncomm_ring,

  obtain ⟨b, c, b_nonzero, c_nonzero, bx_eq_ca⟩
    := left_ore.prop x a x_nonzero a_nonzero,
  have ab_nonzero : (a*b ≠ 0) := begin
    intro ab_zero,
    cases (integral.prop a b ab_zero) with a_zero b_zero,
    exact a_nonzero a_zero, exact b_nonzero b_zero,
  end,
  obtain ⟨k, l, abk_p_lab_eq_one⟩ := weak_fadelian.prop (a*b) ab_nonzero,
  existsi (b*k*x), existsi (l*a*c),
  assoc_rewrite (eq.symm bx_eq_ca),
  rw [← mul_assoc, ← mul_assoc, ← add_mul],
  rw [mul_assoc l a, ← abk_p_lab_eq_one, one_mul],
end
\end{lstlisting}

\end{document}